\documentclass[12pt]{amsart}

\usepackage{amsmath,amssymb}
\usepackage{amsthm}
\usepackage{hyperref}
\usepackage[margin=1in]{geometry}

\numberwithin{equation}{section}

\allowdisplaybreaks

\newtheorem{prop}{Proposition}
\newtheorem{lemma}[prop]{Lemma}

\newtheorem{theorem}[prop]{Theorem}
\newtheorem{cor}[prop]{Corollary}

\newtheorem{definition}{Definition}

\newcommand{\Ric}{\operatorname{Ric}}
\newcommand{\grad}{\nabla}
\newcommand{\D}{\partial}

\newcommand{\tr}{\operatorname{tr}}
\newcommand{\Id}{\operatorname{Id}}

\newcommand{\Sym}{\operatorname{Sym}}

\begin{document}
	
\title[A steady length function for Ricci flow]{A steady length function for Ricci flows}
\author{Joshua Jordan}

\begin{abstract}
	A fundamental step in the analysis of singularities of Ricci flow was the discovery by Perelman of a monotonic volume quantity which detected shrinking solitons in \cite{Perelman}. A similar quantity was found by Feldman, Ilmanen, and Ni in \cite{FIN} which detected expanding solitons. The current work introduces a modified length functional as a first step towards a steady soliton monotonicity formula. This length functional generates a distance function in the usual way which is shown to satisfy several differential inequalities which saturate precisely on manifolds satisfying a modification of the steady soliton equation.
\end{abstract}

\date{\today}	
\maketitle
	
\section{Introduction}
In his paper \cite{Perelman}, Grigori Perelman defines a length functional which is very similar to the Riemannian energy of a curve except that it has a weighting scalar curvature term and the metric is allowed to evolve along the curve. By making clever use of Hamilton's Harnack-type operators and integration by parts, Perelman shows that a modification of the associated distance function has a particular Hessian estimate which implies that it is a subsolution of the $L^2$-conjugate heat equation along the Ricci flow. In fact, he shows that this modified distance function is a solution in precisely the setting of shrinking gradient solitons. This is precisely the condition necessary for the form $\tau^{-\frac{n}{2}}e^{-l(p,\tau)}dV(g_\tau)$ (where $\tau$ is reversed time and $l$ is a modification of the reduced distance) to have a monotone non-decreasing integral along the flow which is constant precisely in the setting of shrinking gradient solitons. This has useful implications for blow-up arguments. In fact, in \cite{Perelman}, Perelman offers a proof of a weakened ``No local collapsing" theorem using only volume monotonicity, a simple differential inequality of the distance function, and some standard Riemannian geometry.
 
A similar reduced length-type quantity for expanding Ricci solitons was introduced by Feldman, Ilmanen, and Ni in \cite{FIN} by drawing analogies to the work of Huisken in \cite{Huisken}. This reduced length is shown to have its own monotone volume which is constant precisely on expanding solitons. This monotone reduced volume in the expanding setting has proved useful in studying rigidity and asymptotics of solutions to Ricci flow.

In short, monotone volume quantities which detect Ricci solitons can have powerful implications, but the case of steady solitons seems largely undeveloped. As such, the author proposes a reduced length functional in this direction.
\begin{definition}
	Given a solution to Ricci flow $(M,g_\tau)$ with maximal existence time $\tau^*$ and two points $(p,s),(q,t)\in M\times [0,\tau^*)$, we define a functional on the space of paths $\gamma$ from $p$ to $q$ in $M$ satisfying $\gamma(s)=p$ and $\gamma(t)=q$
	$$\mathcal{L}(\gamma:(p,s)\to (q,t)) = \int_s^{t} R_{g_{\bar{\tau}}}+\left|\frac{d \gamma}{d\bar{\tau}}\right|^2_{g_{\bar{\tau}}} d\bar{\tau}.$$
\end{definition}
This quantity is very similar to the Perelman and Feldman-Ilmanen-Ni length functionals, and so is a natural choice for a steady length functional. However, this functional is unique in its dependence on both endpoints of the given curve. This will distinctly change the analysis, as the Hessian estimate in \cite{Perelman} relies on having one fixed endpoint.

The associated two-point distance function is then defined as follows.
\begin{definition} \label{DistanceFunction}
	Along a solution to Ricci flow $(M,g_\tau)$ with $\tau^*$ as above, we can define a function $L:(M\times[0,\tau^*)) \times (M \times [0,\tau^*)) \to \mathbb{R}$ by
	$$L((p,s),(q,t)) = \inf\{\mathcal{L}(\gamma:(p,s)\to(q,t))| \gamma\in C^1\}.$$
\end{definition}

This reduced distance satisfies two interesting global differential inequalities which are similar in form to inequalities satisfied by both the Perelman and Feldman-Ilmanen-Ni lengths. However, due to the two-point dependence, the condition for these inequalities on $L$ to saturate becomes a modified version of the usual steady gradient soliton equation. To be more precise, consider the following definition.
\begin{definition}
	$(M,g)$ will be called a twisted gradient soliton provided the following holds for every $(p,s),(q,t)\in M\times [0,\tau^*)$ which may be joined uniquely by a minimizing $\mathcal{L}$-geodesic $\gamma:(p,s)\mapsto (q,t)$.
	$$\grad_{M\times M}^2L\big|_{((p,s),(q,t))}\circ(\Id \oplus \parallel_\gamma)+ \Ric|_{(p,s)}\oplus (-\Ric|_{(q,t)})=0$$
\end{definition} 
This definition in hand, we can now prove the following theorems.
\begin{theorem} \label{MainTheorem}
	Given the $L$ distance defined above on the spacetime of a solution to Ricci flow, we have the following global inequalities in the barrier sense.
	\begin{equation}\label{LEvol}
	\frac{\D}{\D s}L((p,s),(q,t))+\frac{\D}{\D t}L((p,s),(q,t))-2\Box_{M\times M}L((p,s),(q,t))\geq 0
	\end{equation}
	
	\begin{equation}\label{LEllip}
	2\Box_{M\times M} L((p,s),(q,t)) +|\grad_p L|^2((p,s),(q,t))-|\grad_q L|^2((p,s),(q,t))+R_{g_t}(q)-R_{g_s}(p)\leq 0.
	\end{equation}
	Equality is obtained precisely on twisted gradient solitons.
\end{theorem}
The operator $\Box_{M\times M}$ will be defined later, but it is a partial trace of the $M\times M$-Hessian.

A neat corollary of this theorem is the following monotonicity result.
\begin{cor}\label{Monot}
	Let $L$ be the distance function from above, then for $A>0$, the quantity $\inf_{(M\times \{t\})\times (M\times \{t+A\})} L$ is non-decreasing and constant on twisted gradient solitons.
\end{cor}

\section{First Variation of $\mathcal{L}$}

To begin, we will compute the first variation of the $\mathcal{L}$ functional. This is the first place where this functional can be seen as a promising candidate for a steady length functional. The Euler-Lagrange equation is very similar to the Euler-Lagrange equations for both Perelman and Feldman-Ilmanen-Ni's functionals, except without the $T/\tau$ term.

\begin{prop}\label{FirstVar}
Let $M$ be a smooth manifold with $g_\tau$ a one-parameter family of metrics solving Ricci flow on $[0,\tau^*)$. Letting $\gamma:(p,r)\to (q,t)$ be a smooth path, $V$ a smooth vector field along $\gamma$, and $T=\frac{d}{d\tau}
\gamma(\tau)$ . Then,
$$\delta_{\gamma}\mathcal{L}(V) = 2\langle      V,T\rangle|_{r}^{t}+\int_{r}^{t}\langle V, \grad R-2\grad_T T+4\Ric(T)^\sharp\rangle d\bar{\tau}.$$
Furthermore, the associated geodesic equation is
$$\grad_T T - 2\Ric(T)^\sharp=\frac{1}{2}\grad R.$$

\begin{proof}
To begin, we note the identity
$$\frac{d}{d\bar{\tau}}\langle V,T \rangle = \langle \grad_TV,T \rangle+\langle V,\grad_T T\rangle-2\Ric(V,T)$$
where $T$ is the $\tau$-velocity of $\gamma$. 

Let $\alpha(s,\bar{\tau})=\gamma_s(\bar{\tau})$ with variation field $V$ and tangent $T(\bar{\tau})= \frac{d}{d\bar{\tau}}\alpha(0,\bar{\tau})$. Then the first variation can be found
\begin{align*}
\delta_\gamma \mathcal{L} (V)=& \int_r^{t}(\langle \grad R, V\rangle+\frac{d}{ds}\bigg|_{s=0}\left|\frac{d\alpha}{d\bar{\tau}}\right|^2)d\bar{\tau}\\
=& \int_r^{t}(\langle \grad R, V\rangle+2\langle D_s|_{s=0}\frac{d\alpha}{d\bar{\tau}},T \rangle)d\bar{\tau}\\
=& \int_r^{t}(\langle \grad R, V\rangle+2\langle \grad_T V,T \rangle)d\bar{\tau}\\
=&\int_r^t\langle \grad R,V\rangle +2(\frac{d}{d\bar{\tau}}\langle V,T\rangle - \langle V,\grad_T T\rangle + 2 \Ric(V,T))d\bar{\tau} \\
=&\ 2\langle V,T\rangle|_r^{t}+\int_r^{t}\langle \grad R-2\grad_T T+4\Ric(T)^\sharp, V\rangle d\bar{\tau}.
\end{align*} 
This is precisely the result and the geodesic equation is picked out as usual.
\end{proof}
\end{prop}

\section{Second Variation of $\mathcal{L}$}

The second variation makes use of the following lemma found in Kleiner-Lott \cite{KleinerLott}. The proof is provided here for convenience.
\begin{lemma}\label{ChrisEvol}
	Let $\gamma$ be a curve through the spacetme of a solution to Ricci flow with parameter $\tau$ and $T = \frac{\D}{\D \tau}\gamma$. Then, the following formula holds
	$$\frac{d}{d\tau}\langle \grad_T V, V \rangle= |\grad_T V|^2+\langle \grad_T\grad_T V,V \rangle-2\Ric(\grad_T V, V)-(\grad_T\Ric)(V,V).$$
	\begin{proof}
	We will prove this in the more general setting where $\frac{\partial g}{\partial \tau}=h \in \Sym^2(T^*M)$. Also, we will make use of the notation
	$$\dot{\grad}_XY =  X^iY^l\frac{\D \Gamma_{il}^k}{\D \tau} \D_k.$$
	
	As the result of Lemma \ref{ChrisEvol} is tensorial, we may compute at the center of a normal coordinate neighborhood. 
		\begin{align*}
		\frac{d}{d\tau}\langle \grad_X Y,Z \rangle =&\ \frac{d}{d\tau}\left[X^iZ^k(Y^j_{,i}+Y^l\Gamma_{il}^j)g_{jk}\right]\\
		=&\ \frac{dg_{jk}}{d\tau}X^iY^j_{,i}Z^k+\frac{d}{d\tau}(X^iZ^k)Y^k_{,i}+X^iZ^k\left(\frac{dY^k_{,i}}{d\tau}+Y^l\frac{d}{d\tau}\Gamma_{il}^k\right)\\
		=&\ h_{jk}X^iY^j_{,i}Z^k+\dot{X}^iY^k_{,i}Z^k+X^i\dot{Y}^k_{,i}Z^k+X^iY^k_{,i}\dot{Z}^k+X^iY^lZ^k\frac{\partial}{\partial\tau}\Gamma_{il}^k+X^iY^lZ^kT^p\Gamma_{il,p}^k\\
		=&\ h(\grad_X Y,Z)+\langle \grad_X Y, \grad_T Z\rangle+\langle\dot{\grad}_XY,Z \rangle + T^p[(X^iY_{,i}^k)_{,p} + X^iY^l\Gamma_{il,p}^k]Z^k \\
		=&\ h(\grad_X Y,Z)+\langle \grad_X Y, \grad_T Z\rangle+\langle\dot{\grad}_XY,Z \rangle + T^p[X^iY_{,i}^k + X^iY^l\Gamma_{il}^k]_{,p}Z^k \\
		=&\ h(\grad_X Y,Z)+\langle \grad_X Y, \grad_T Z\rangle+\langle\dot{\grad}_XY,Z \rangle + T^p[\grad_XY^k]_{,p}Z^k \\
		=&\ h(\grad_X Y,Z)+\langle \grad_X Y, \grad_T Z\rangle+\langle\dot{\grad}_XY,Z \rangle + \langle\grad_T\grad_X Y,Z \rangle.
		\end{align*}
		Specializing to the setting of Ricci flow gives the result.
	\end{proof}
\end{lemma}

It will be noticed that in the following proposition, we once again have a result very similar to those found by Perelman and Feldman-Ilmanen-Ni, except for a term like $(\grad_T V)/\tau$.
\begin{prop} \label{SecondVar}
	Under the same conditions as Propositon \ref{FirstVar}, the total second variation of $\mathcal{L}$ (denoted by $Q$) is given by
	\begin{multline} \label{QC2}
	Q(V,V) =\ 2\langle \grad_T V,V\rangle|_r^{t}+\int_r^{t}\langle V, \grad^2 R(V)^\sharp -2Rm(V,T)T\\
	+4(\grad_V\Ric)(T)^\sharp-2\grad_T\grad_T V + 4\Ric(\grad_T V)^\sharp\rangle d\bar{\tau}.
	\end{multline}
	Equivalently,
	\begin{equation}\label{QC1}
	Q(V,V)=\int_r^{t} 2|\grad_T V|^2+2(\grad_T\Ric)(V,V)+\grad^2R(V,V)-2Rm(V,T,T,V)-4(\grad_V\Ric)(T,V)d\bar{\tau}.
	\end{equation}
	Furthermore, we find the following Jacobi equation.
	$$\grad_T\grad_TV + Rm(V,T)T - 2(\grad_V \Ric)(T)^\sharp - 2\Ric(\grad_T V)^\sharp-\frac{1}{2}\grad^2 R(V)^\sharp=0$$
\begin{proof}
We will prove Equation \ref{QC2} by taking the same variation $\alpha$ as in Proposition \ref{FirstVar}, differentiating twice, and applying the Symmetry Lemma and Lemma \ref{ChrisEvol}. 
\begin{align*}
\delta^2_\gamma \mathcal{L}(V,V)=& \int_a^{b}VVR+2\frac{d}{ds}|_{s=0}(\langle D_s\frac{d\alpha}{d\bar{\tau}},\frac{d\alpha}{d\bar{\tau}} \rangle)d\bar{\tau}\\
=&\int_a^{b}VVR+2\langle \grad_V\grad_VT, T\rangle +2|\grad_VT|^2d\bar{\tau}\\
=&\int_a^{b}VVR-2Rm(V,T,T,V)\\
&+2\left(\frac{d}{d\tau}\langle \grad_V V,T\rangle -\langle \grad_V V\grad_T T\rangle +2\Ric(\grad_V V,T)+2(\grad_V \Ric)(V,T)\right.\\
&\left.-(\grad_T \Ric)(V,V)\right)+2\left(\frac{d}{d\tau}\langle \grad_T V,V \rangle -\langle \grad_T \grad_T V,V\rangle+2\Ric(\grad_TV,V)\right.\\
&\left.+(\grad_T\Ric)(V,V) \right) d\bar{\tau}\\
=&\ 2\langle\grad_V V,T \rangle|_a^{b}+ 2\langle \grad_T V,V\rangle|_a^{b}+\int_a^{b}[VVR-2Rm(V,T,T,V)-2\langle \grad_V V,\grad_T T\rangle\\ 
&+ 4\Ric(\grad_V V, T)+4(\grad_V\Ric)(V,T)-2\langle \grad_T\grad_T V, V\rangle +4\Ric(\grad_T V,V)]d\bar{\tau}.
\end{align*}
Thus the tensorial second variation is given by
\begin{align*}
Q(V,V)=&\ \delta^2_\gamma \mathcal{L}(V,V)-\delta_\gamma\mathcal{L}(\grad_VV)\\
=&\ 2\langle \grad_T V,V\rangle|_a^{b}+\int_a^{b}\langle V, \grad^2 R(V)^\sharp -2Rm(V,T)T+4(\grad_V\Ric)(T)^\sharp-2\grad_T^2 V+ 4\Ric(\grad_T V)^\sharp\rangle d\bar{\tau}.
\end{align*}
Equation \ref{QC1} follows from integrating by parts.
\end{proof}
\end{prop}

\section{First Derivatives of $L$}
In this section, we compute first derivatives of $L$ using the variations of $\mathcal{L}$ computed above.

\begin{prop} \label{Gradient}
	The gradients of $L$ (as defined in Definition \ref{DistanceFunction}) are given by the following 
	\begin{align*}
	\grad_p L((p,r),(q,t)) =&-2T(r)\\
	\grad_q L((p,r),(q,t)) =&\  2T(t).
	\end{align*}
\begin{proof}
Let $c(s)$ be a smooth curve along $M$ w/ $c(0)=p$ and $c^\prime(0)=v\in T_{p}M$. Then fixing $a$, define $\alpha$ a variation through geodesics s.t. $\alpha(s,r)=c(s)$ and $\alpha(s,t)=q$. Define $V = \frac{d}{ds}\big|_{s=0}\alpha$. Clearly, $V(r)=v$ and $V(t)=0$. Then, by the first variation formula,
$$\langle \grad_p L((p,r),(q,t)), v\rangle = \frac{d}{ds}\big|_{s=0}\mathcal{L}(\alpha) = 2\langle V,T \rangle|_r^{t} = -2\langle v,T \rangle(r).$$
Similarly, 
$$\langle\grad_qL((p,r),(q,t)), v \rangle= 2\langle T,v\rangle(t).$$
\end{proof}
\end{prop}

\begin{cor}\label{GradNorms}
	As direct consequences of Proposition \ref{Gradient}, we can compute the following quantities.
	\begin{align*}
	\langle \grad_pL, T \rangle ((p,r),(q,t)) =& -2|T|^2(r)\\
	|\grad_pL|^2((p,r),(q,t))=&\ 4|T|^2(r)\\
	\langle\grad_qL, T \rangle((p,r),(q,t))=&\ 2|T|^2(t)\\
	|\grad_q L|^2((p,r),(q,t))=&\ 4|T|^2(t).
	\end{align*}
\end{cor}
\qed

With knowledge of the gradients, we can compute the time-derivatives by making use of the chain rule.
\begin{prop}\label{LEvol2}
	The following evolutions hold 
	$$\frac{\partial L}{\partial s}((p,s),(q,t))=-R_{g_s}(p)+|T|^2(s),$$
	$$\frac{\partial L}{\partial t}((p,s),(q,t))=R_{g_t}(q)-|T|^2(t).$$
\begin{proof}
Let $\gamma(\bar{\tau}):[s,t]\to M$ be an $L$-minimizing geodesic with $\gamma(s)=p$ and $\gamma(t)=q$. Then, by the chain rule, we have 
$$\frac{\D L}{\D s}((p,s),(q,t)) = \frac{d\mathcal{L}}{ds}(\gamma:(p,s)\to (q,t))-\langle \grad_p L,T\rangle ((p,s),(q,t))$$
where $\frac{d}{d\bar{\tau}}\gamma(\bar{\tau})=T$.

The first term can be computed by the fundamental theorem.
$$\frac{d}{ds}\mathcal{L}(\gamma:(p,s)\to (q,t)) = \frac{d}{ds}\int_s^{t}(R_{g_{\bar{\tau}}}+|T|^2)d\bar{\tau} = -R_{g_s}(p)-|T|^2(s).$$

Invoking Corollary \ref{GradNorms},
\begin{align*}
\frac{\partial L}{\partial s}((p,s),(q,t)) = & \frac{d\mathcal{L}}{ds}(\gamma:(p,s)\to (q,t))-\langle\grad_pL,T \rangle((p,s),(q,t))\\
=&-R_{g_s}(p)+|T|^2(s).
\end{align*}
The argument proceeds similarly for the $t$-partial.
\end{proof}
\end{prop}

\section{Hessian Estimate and Proof of Theorem}

\begin{prop} \label{HessianEst}
Let $I$ be an interval in $\mathbb{R}$ and $L:(M\times I) \times (M\times I)\to \mathbb{R}$ as defined above. Fix a minimizing geodesic $\gamma:[s,t]\to M$ with $\gamma(s)=p$ and $\gamma(t)=q$. Let $V$ be a vector field along $\gamma$ solving 
$$\grad_T V = \Ric(V)^\sharp.$$ 
Then, the following Hessian estimate holds
$$\grad^2_{M\times M} L (V(s)\oplus V(t),V(s)\oplus V(t))\leq 2\Ric(V,V)|_{s}^t-\int_s^{t}H(T,V)d\bar{\tau}.$$
Where 
\begin{multline*}
H(T,V)=2\Ric_\tau(V,V)+4[(\grad_T\Ric)(V,V)-(\grad_V\Ric)(V,T)]\\
+2|\Ric(V)|^2
+2Rm(V,T,T,V)-\grad^2 R(V,V)
\end{multline*}
The equality holds precisely when $V$ is $\mathcal{L}$-Jacobi.

\begin{proof}
Take $V$ solving $\grad_TV=\Ric V^\sharp$ on $[s,t]$. Then, by a standard argument, we find that
$$\grad_{M\times M}^2L(V(s)\oplus V(t),V(s)\oplus V(t))\leq Q(V,V)$$
with equality holds precisely when $V$ is $\mathcal{L}$-Jacobi.

We must also utilize
\begin{align*}
\Ric(V,V)|_s^{t}=& \int_s^{t}\frac{d}{d\bar{\tau}}[\Ric(V,V)]d\bar{\tau}\\
=&\int_s^{t}\Ric_{\bar{\tau}}(V,V) + (\grad_T\Ric)(V,V)+2\Ric(\grad_TV,V)d\bar{\tau}.
\end{align*}
This gives 
\begin{align*}
Q(V,V) =& 2\Ric(V,V)|_{s}^t+2\int_s^{t}\Ric_{\bar{\tau}}(V,V) + (\grad_T\Ric)(V,V)+2\Ric(\grad_TV,V)d\bar{\tau}+Q(V,V)\\
=&2\Ric(V,V)|_{s}^t-\int_s^{t}H(T,V)d\bar{\tau}.
\end{align*}
where $H(T,V)$ is as was suggested.
\end{proof}
\end{prop}

\begin{prop}
	Let $\gamma:(p,s)\to (q,t)$ smoothly. Fix $\{e_i\}$ an orthonormal frame along $\gamma$. Then,
	$$\sum_i H(T,e_i)=\frac{d}{d\tau}(R+|T|^2)$$
	on $(s,t)$.
	\begin{proof} 
		As all quantities are tensorial, we may fix a time $\bar{\tau}\in(s,t)$ and compute there.
		\begin{align*}
		\sum_i H(T,e_i)=& 2R_\tau +4[\grad_T R-\operatorname{div}\Ric(T)]+2|\Ric|^2+2\Ric(T,T)-\Delta_g R\\
		=&R_\tau+2\langle \grad R,T\rangle +2\Ric(T,T).
		\end{align*}
		This follows quickly from the contracted Bianchi identity and the definition of Ricci curvature.
		
		To notice the right-hand side, compute using the geodesic equation.
		\newline
		\begin{align*}
		\frac{d}{d\tau}(R+|T|^2)=&\ R_\tau +\langle \grad R,T\rangle + 2\langle \grad_T T,T\rangle -2\Ric(T,T)\\
		=&\ R_\tau+ 2\langle \grad R, T\rangle + 2\Ric(T,T)
		\end{align*}
		This is precisely what was expected.
	\end{proof}
\end{prop}

For the statement of the main theorem, a definition is needed. As it has only been shown that the Hessian is estimated along vector fields solving a first order ODE, we will have estimates for the trace only on an $n$-dimensional subspace of $T_pM\oplus T_q M$. 

\begin{definition}
	Let $\gamma:(p,s)\to (q,t)$ min geo. Then we can define an operator $\parallel_\gamma:T_pM \to T_qM$ as follows. Let $v\in T_p M$, then by the standard linear first order ODE theory, we can solve
	$$\begin{cases}
	\grad_T V= \Ric(V)^\sharp \\
	V(s)=v
	\end{cases}.$$
	Then $\parallel_\gamma v := V(t)\in T_qM$.
\end{definition}

\begin{definition}
	For $f\in C^2((M\times I)\times (M\times I))$, take $\gamma: (p,s)\to (q,t)$ a minimizing $\mathcal{L}$-geodesic. Then we can define the operator $\Box_{M\times M}$ as follows.
	$$\Box_{M\times M} f((p,s),(q,t))= \tr_{g} \left[\grad^2_{M\times M}f \left((\Id_{T_pM} \oplus \parallel_\gamma)\bullet, (\Id_{T_pM} \oplus \parallel_\gamma)\bullet\right)\right].$$
\end{definition}
Notice that this operator is well-defined away from the mutual cut locus, as it is defined invariantly and the choice of geodesic is unique therein.

	\begin{proof}[Proof of Theorem \ref{MainTheorem}, Inequality \ref{LEvol}]
	Take an orthonormal basis $\{e_i\}$ at $p$. Then, extend these vectors to fields $\{E_i(\bar{\tau})\}$ along $\gamma:(p,s)\mapsto (q,t)$ (a minimizing geodesic) by solving the equation
	$$\grad_T E_i = \Ric(E_i)^\sharp.$$
	Notice that since 
	$$\frac{d}{d\tau}\langle E_i,E_j \rangle = 0$$
	this differential equation extends the orthonormal basis at $p$ to an orthonormal frame along $\gamma$ . This implies that $\{E_i(\bar{\tau})\}$ evaluates to an orthonormal basis at $q = \gamma(t)$. Tracing inequality \ref{HessianEst}  along the basis $(\frac{1}{\sqrt{2}}(E_i(s)\oplus E_i(t)))$ gives 
	\begin{align*}
	\Box_{M\times M}L =& \sum_i \grad^2_{M\times M} L (\frac{1}{\sqrt{2}}(E_i(s)\oplus E_i(t)),\frac{1}{\sqrt{2}}(E_i(s)\oplus E_i(t)))\\
	\leq&\ \frac{1}{2} \sum_i \left(2\Ric(E_i,E_i)|_s^t-\int_s^tH(T,E_i)\right)\\
	\leq&\  R(q)-R(p)-\frac{1}{2}\int_s^{t}\frac{d}{d\bar{\tau}}(R+|T|^2)d\bar{\tau}\\
	=&\ R(q)-R(p)-\frac{1}{2}(R(q)+|T|^2(t))+\frac{1}{2}(R(p)+|T|^2(s))\\
	=&\ \frac{1}{2}R(q)-\frac{1}{2}R(p)+\frac{1}{2}|T|^2(s)-\frac{1}{2}|T|^2(t)\\
	=&\frac{1}{2}\left(\frac{\partial L}{\partial s}+\frac{\D L}{\D t}\right).
	\end{align*}
	This means that $L$ satisfies the differential inequality
	$$(\frac{\D}{\D s}+\frac{\D}{\D t}-2\Box_{M\times M})L\geq 0.$$

	Fixing $(x,y)$ and $(z,w)$ and we will construct barriers there by analogy to \cite{KleinerLott}. Choose some minimizing geodesic $\gamma:(x,y)\to (z,w)$ and let $\epsilon>0$. Then define 
	$$f_\epsilon(((p,s),(q,t)) = L((p,s),(\gamma(y+\epsilon),y+\epsilon))+L((\gamma(y+\epsilon),y+\epsilon),(q,t)).$$
	Then, by uniqueness of geodesics $f_\epsilon((x,y),(z,w))=L((x,y),(z,w))$ and $f_\epsilon\geq L$ near $((x,y),(z,w))$ by the analogue of the triangle inequality. Therefore, by computations similar to the previous
		\begin{align*}
		\frac{\D}{\D s}f_\epsilon((x,y),(z,w))=&\ -R_{g_y}(x)+\left|\frac{d\gamma}{d\tau}\right|^2(y)\\
		\frac{\D}{\D t}f_{\epsilon}((x,y),(z,w))=&\  R_{g_w}(z)-\left|\frac{d \gamma}{d\tau}\right|^2(w).
		\end{align*}
		
	Computing the box operator is a little trickier, since it is not immediately well-defined on the cut-locus. With $\gamma$ constructed as above, we define a new family of operators
	\begin{multline*}
	\Box_{\gamma,\epsilon}f((p,s),(q,t)) = \tr_g\left[ \grad^2_{(p,s)} f \left((\operatorname{Id}_{T_pM}\oplus\parallel_{(p,s)\to (\gamma(y+\epsilon),y+\epsilon)}\circ \parallel_{(\gamma(y+\epsilon),y+\epsilon)\to (q,t)})\bullet,\right.\right.\\\left.\left.(\operatorname{Id}_{T_pM}\oplus\parallel_{(p,s)\to (\gamma(y+\epsilon),y+\epsilon)}\circ \parallel_{(\gamma(y+\epsilon),y+\epsilon)\to (q,t)})\bullet\right)\right]
	\end{multline*}
	where the $\parallel$ operators are defined as usual on the geodesic segments connecting the indicated points. Then, by applying the Hessian estimates, we find
	\begin{align*}
	2\Box_{\gamma,\epsilon}f_\epsilon((x,y),(z,w)) =&\ 2[\Box_{\gamma,\epsilon}L((x,y),(\gamma(y+\epsilon),y+\epsilon))+\Box_{\gamma,\epsilon}L((x,y),(\gamma(y+\epsilon),y+\epsilon))]\\
	\leq&\ \frac{\D f_\epsilon}{\D t}((x,y),(z,w))+\frac{\D f_\epsilon}{\D s}((x,y),(z,w)).
	\end{align*}
	Since this is sufficient to give us a maximum principle on the cut-locus, this proves the inequality in the ``barrier-sense".
	
	Notice that in the above, saturation occurs when all $E_i$'s are Jacobi.
	But when all $E_i$'s are also Jacobi, we find that every field solving $\grad_TV =\Ric(V)^\sharp$ must be Jacobi. So, taking $v=v^ie_i\in T_pM$ and extending by the equation to $V$ along $\gamma$ (a geodesic near $p$)
	$$0=\frac{d|V|^2}{d\tau} = 2\langle \grad_T V,V\rangle - 2\Ric(V,V).$$
	But evaluating at $q$ then gives that
	\begin{align*}
	0=&\ 2\grad_{M\times M}^2L(V(s)\oplus V(t),V(s)\oplus V(t))-2\Ric|_q(V(t),V(t))+2\langle \grad_T V,V\rangle (s)\\
	=&\ 2\grad_{M\times M}^2L(V(s)\oplus V(t),V(s)\oplus V(t))-2\Ric|_q(V(t),V(t))+2\Ric|_p(V(s),V(s)).
	\end{align*}
	for every $v\in T_pM$. Since both of these are symmetric tensors, they are uniquely determined by their actions on the diagonal subspace. So, $$\grad^2_{M\times M} L \circ (\Id \oplus \parallel_\gamma) + \Ric_p \oplus (-\Ric_q)=0.$$
	This is exactly the twisted gradient soliton equation.
	\end{proof}

\begin{proof}[Proof of Corollary \ref{Monot}]
	Notice that, for any $t$, by compactness there is a point $((x,t),(y,t+A))\in M\times\{t\} \times M\times \{t+A\}$ s.t. $\inf_{M\times\{t\}\times M\times \{t+A\}}L = L((x,t),(y,t+A))$. So, by applying the differential inequality at this point,
	$$2\frac{\D}{\D t}L((x,t),(y,t+A))\geq 2\frac{\D}{\D t}L((x,t),(y,t+A))-2\Box L((x,t),(y,t+A))\geq 0,$$
	with the first inequality due to non-negativity of the Hessian at a minimum and the last inequality due to Theorem \ref{MainTheorem}. Thus, the infimum is non-decreasing in $t$. Finally, by the prior theorem, the differential inequality used above is an equality on twisted gradient solitons, giving the result.

\end{proof}

	\begin{proof}[Proof of Theorem \ref{MainTheorem}, Inequality \ref{LEllip}]
		\begin{align*}
		\Box L \leq& \frac{1}{2}R_{g_t}(q)-\frac{1}{2}R_{g_s}(p)+\frac{1}{2}|T|^2(s)-\frac{1}{2}|T|^2(t)\\
		\leq& \frac{1}{2}R_{g_t}(q)-\frac{1}{2}R_{g_s}(p)+\frac{1}{2}|\grad_pL|^2-\frac{1}{2}|\grad_q L|^2(t).
		\end{align*}
			
		A similar computation holds with $f_\epsilon$ as constructed above in the proof of \ref{LEvol}, and so this inequality also holds globally in the barrier sense.
	\end{proof}

\bibliographystyle{acm}

\end{document}